\documentclass{amsart}
\usepackage{comment}
\usepackage{euscript,graphicx,epstopdf,amscd,amsgen,amsfonts,amssymb,latexsym,amsmath,amsthm,graphicx,mathrsfs,times,color,overpic,bbm}

\usepackage{xcolor}
\definecolor{forestgreen}{rgb}{0.0, 0.27, 0.13}

\usepackage{chngcntr}%
\counterwithin{equation}{section}%

\usepackage{mathrsfs}  
\usepackage{psfrag}
\usepackage{mathtools}
\usepackage{color}
\usepackage{todonotes}
\usepackage{enumitem}\usepackage{soul}
\usepackage{hyperref}
\hypersetup{colorlinks=true, citecolor=blue, linkcolor=blue, hypertexnames=false}

\theoremstyle{plain}%
\newtheorem{main}{Theorem}%

\newtheorem{theorem}{Theorem}[section]
\newtheorem{lemma}[theorem]{Lemma}
\newtheorem{corollary}[theorem]{Corollary}
\newtheorem{proposition}[theorem]{Proposition}
\newtheorem{claim}[theorem]{Claim}
\newtheorem{definition}[theorem]{Definition}

\newtheorem{remark}[theorem]{Remark}

\def\bN{\mathbb{N}}

\def\bR{\mathbb{R}}

\def\bX{\mathbb{X}}
\def\bT{\mathbb{T}}

\def\cA{\mathcal{A}}
\def\cD{\mathcal{D}}
\def\cF{\mathcal{F}}
\def\cG{\mathcal{G}}
\def\cM{\mathcal{M}}
\def\cP{\mathcal{P}}

\def\cS{\mathcal{S}}
\def\cU{\mathcal{U}}
\def\cV{\mathcal{V}}

\def\Diff{{\rm Diff}}
\def\dim{{\rm dim}}

\numberwithin{equation}{section}

\DeclareMathSymbol{\varnothing}{\mathord}{AMSb}{"3F}
\title[Robustness of  equilibrium states]{Robustness and uniqueness of equilibrium states for certain partially hyperbolic systems}
\author[J. Mongez]{Juan Carlos Mongez}
\address{Instituto de Matem\'atica, Universidade Federal do Rio de Janeiro, Cidade Universit\'aria - Ilha do Fund\~ao, Rio de Janeiro 21945-909,  Brazil}
\email{jmongez@im.ufrj.br}

\author[M. Pacifico]{Maria Jose Pacifico}
\address{Instituto de Matem\'atica, Universidade Federal do Rio de Janeiro, Cidade Universit\'aria - Ilha do Fund\~ao, Rio de Janeiro 21945-909,  Brazil}
\email{pacifico@im.ufrj.br}

\begin{document}

\begin{abstract}
We prove  that if $f$ is a $C^{1+}$ partially hyperbolic diffeomorphism satisfying certain conditions then there is a
$C^1$-open neighborhood  $\cA$ of $f$ so that every $g\in \cA\cap \operatorname{Diff}^{1+}(M)$ has  a unique equilibrium state. 
\end{abstract}

\thanks{MJP and JCM were partially supported by CAPES-Finance Code 001. 
JCM was partially supported by FAPERJ Grant (Bolsa Nota 10) No. E-26/202.301/2022(276542).
MJP was partially supported by FAPERJ Grant CNE No. E-26/202.850/2018(239069), Grant CNPq-Brazil No. 307776/2019-0 and PRONEX Dynamical Systems E-26/010.001252/2016.}
\footnote{Corresponding author: M. J. Pacifico}
\maketitle
\tableofcontents

\section{Introduction}
 In the context of dynamical systems, for  continuous  maps $f$ defined in  compact metric spaces, the topological entropy $h_{top}(f)$ provides a measure of the complexity of a system by quantifying the growth rate of the number of distinct distinguishable orbits. 
 The pressure $P(f, \phi)$ is a weighted version of the topological entropy $h_{top}(f)$, where the weights are determined by a continuous real function $\phi$, called a potential.
 Then, if $\phi$ is identically zero the topological pressure coincides with the topological entropy.
While  topological entropy quantifies the complexity of a system,
the pressure for a potential describes the balance between the energy and the entropy of the system and can be used to predict its thermodynamical properties. 
The topological pressure is connected to the thermodynamical formalism through the Variational Principle, which relates the topological entropy to the metric entropy of the system,
which measures the complexity of the system from the measure theoretical point of view.
In other words, the Variational Principle provides a way to calculate the topological entropy of a dynamical system by considering its invariant measures and their corresponding entropies.

 To be more precise, denoted by $\cM_f$ the set of $f$-invariant probability measures, the variational principle  establishes   the following  relation 
\begin{equation}
P(f,\phi) = \sup\{h_\mu(f) + \int \phi d\mu : \mu \in \cM_f\}
\end{equation}
An invariant probability $\mu$ is  an {\em equilibrium state} (for short, EE) for the potential $\phi$ if it achieves the supremum in the above equation, that is, if $h_\mu(f) + \int \phi d\mu = P(f,\phi)$.

In particular, since  the topological  pressure coincides with the topological entropy when
the potential $\phi$ is identically zero, the variational principle establishes 
\begin{equation}\label{e-variacional-entropia}
h_{top}(f) = \sup\{h_\mu(f): \mu \in \cM_f\}.
\end{equation}
We say that a measure $\mu$ achieving the supremum in equation (\ref{e-variacional-entropia}) is a {\em measure of maximal entropy} (for short, MME ). 

Maximal entropy  measures  and equilibrium states are key concepts in formal thermodynamical and dynamical systems theory since they provide a way to quantify the degree of chaos in a system  and predict its long-term behavior. 
Thus, measures of maximum entropy and equilibrium states are essential tools to understand the behavior of complex systems and predict their evolution over time.
In particular, a maximal entropy measure has a deep connection with statistical physics and stochastic processes (such as Markov chains) and is closely related to various geometric properties such as the growth rate of closed orbits, as observed by Margulis in his pioneer work \cite{mar70}.

The applicability of these tools to deduce the asymptotic behavior of the system depends on
the existing number of such measures. The optimal case is when there is only one. 
 It is not difficult to construct examples of partially hyperbolic systems with positive and finite entropy which have an infinite number of measures of maximal entropy.
 In these cases, studying the system through its measures of maximal entropy becomes much more difficult.
 Therefore, determining whether a system has a unique measure of maximum entropy has been one of the challenges in dynamical systems theory. 

The study of equilibrium states was initiated by Sinai, Ruelle and Bowen in the 1970s. Sinai was a pioneer in investigating the existence and finiteness of equilibrium states for Anosov diffeomorphisms for continuous H\"older potentials \cite{sinai72}.
Later, Ruelle and Bowen  expanded this approach to include uniformly hyperbolic (Axiom A)  systems \cite{Bowen71, Ruelle68, Ruelle78}.

Bowen,  in \cite{Bowen74} provides a criterion to determine whether a dynamical system has a unique equilibrium state for a given potential function:
for homeomorphisms $f$ defined on a compact metric space $\bX$, if $(\bX,f)$ is an expansive system with specification and the potential satisfies a certain regularity condition (known as the Bowen property) then the system has a unique equilibrium state.  
This criterion is easily applied to Axiom A diffeomorphisms 
 but it is not always applicable to partially hyperbolic systems \cite{SNVPY2016}.

If an Anosov diffeomorphism has a unique measure of maximum entropy, then it is possible to prove that diffeomorphisms close to it also have a unique measure of maximum entropy. However, this is not always possible for partially hyperbolic  systems, for instance,  for a skew product of an Anosov by an irrational rotation.

Climenhaga and Thompson improved the Bowen criterion in \cite{climThom2016} and showed that the same conclusion holds  using weaker non-uniform versions of specification, expansivity, and Bowen property over a significant set of orbit segments.
This criterion applies to certain  systems like the Bonatti-Viana family of diffeomorphisms in \cite{CFT2018} and Ma\~n\'e's derived from Anosov diffeomorphism on $\bT^3$
 in \cite{CFT2019}.
 
This criterion does not apply to flows with singularities, as it happens to Lorenz-type attractors.
To address this issue, Pacifico et al. developed an improved version of the CT criterion  \cite{PYY2022}, and using it they were able to prove that Lorenz-like attractors in any dimension have a unique maximum entropy measure \cite{PYYLorenz2022}.

Determining which partially hyperbolic systems have a unique measure of maximal entropy is a challenging task. However, using their criterion, Climenhaga and Thompson 
following the approach in  \cite{CFT2019}, were able to prove that partially hyperbolic systems with central dimension one and such that every MME has 
 central Lyapunov exponent negative and  the unstable foliation is minimal
have a unique measure of maximal entropy \cite{climThom2021}. 
Since the minimality of the unstable foliation and, more importantly, the assumption about the central Lyapunov exponent are not open conditions,  
the robustness of the uniqueness of the maximum entropy measure does not follow directly from this work.

We note that while the minimality condition implies $\varepsilon$-minimality (see Definition \ref{defi-epsilon-minimality}), allowing the application of \cite{PYY2022} to conclude that the uniqueness of the equilibrium state is robust, the assumption about the central Lyapunov exponent cannot be easily replaced by another condition that would yield the desired result. To overcome this issue, we explore the properties of unstable and stable entropy, denoted as $h^u(f)$ and $h^s(f)$, respectively, as introduced in \cite{yang2016, HHHY2017} (see Definition  \ref{def-u-s-entropia}). By replacing the aforementioned assumption about the central Lyapunov exponent with the condition $h^u(f)-h^s(f)>0$, we were able to prove the following result:

\begin{main}\label{teo-main}
Let $f:M\to M$ be a $C^{1+}$ partially hyperbolic diffeomorphism of a compact manifold $M$ with $TM=E^u \oplus E^c \oplus E^s $ and $\phi:M\to \mathbb{R}$ a H\"older continuous potential. Assume that $dim E^c=1$ and the unstable foliation $\mathcal{F}^u(f)$ 
is minimal.
If $h^u(f)-h^s(f)>\sup \phi-\inf \phi\geq 0$ then there exists a $C^1$ neighborhood  $\mathcal{U}$ of $f$ so that  $(g,\phi)$ has a unique equilibrium state for every $C^{1+}$ diffeomorphism $g\in \mathcal{U}$.
\end{main}

\subsubsection*{Organization of the paper} 
In Section \ref{s-preliminar} we provide the readers with preliminaries on a criterion for uniqueness of measures of maximal entropy, set the notations and definitions, which are taken from \cite{climThom2016, PYY2022}. In particular, in this section, we define stable $h^s(f)$ and unstable $h^u(f)$ entropy for partially hyperbolic systems, taken from \cite{yang2016, HHHY2017}. In Section \ref{s-consequences} we establish the consequences  of
 $h^u(f) - h^s(f) > 0$. 
In Section \ref{s-prova-teoA} we provide the proof of Theorem \ref{teo-main} showing
   that every $g$ in a neighborhood $C^{1+}$ of a partially hyperbolic diffeomorphism $f$ with 
   a one-dimensional central bundle, such that the unstable foliation $\mathcal{F}^u$ is minimal and such that $h^ u(f ) - h^s(f) > 0$ satisfies all the requirements of Theorem \ref{criterioCT}, finishing the proof.
   Finally, in Section \ref{s-example} we show that the classical class of examples of a non-hyperbolic, robustly transitive diffeomorphism, isotopic to an Anosov automorphism and derived from an Anosov system, as constructed by Ma\~n\'e in \cite{M78}, satisfies the hypotheses of our theorem.

\section{Preliminaries}\label{s-preliminar}
\subsection{ A criterion to the uniqueness of MME }\label{ss-PYY}

In this section, we provide a criterion developed in \cite{PYY2022} to obtain the uniqueness of
MME. 
 For the convenience of those who are familiar with \cite{climThom2016} or \cite{PYY2022}, the notations here are the same. So it is safe to skip most of this Section and move on to section \ref{ss-PH}. 

\subsubsection{Topological pressure}\label{ss-Top-pressure}
Let $M$ be a compact metric space and $f:M\to M$ a homeomorphism. The nth Bowen metric associated with $f$ is defined by
\begin{equation*}
    d_n(x,y):=\sup\{d(f^k(x),f^k(y))|0\leq k<n\}.
\end{equation*}

The $n$-Bowen ball with radius $\varepsilon>0$ centred at $x\in M$ is given by
\begin{equation*}
B_n(x,\delta):=\{y\in M:d_n(x,y)<\delta\}.  
\end{equation*}

Given $\delta> 0, \, n\in\mathbb{N},$ a set $E\subset M$
 is $(n,\delta)$-separated if for every pair of distinct  points $x,y\in M$ it holds $d_n(x,y)>\delta$.

Given a continuous potential $\phi:M\to \mathbb{R}$, write $\Phi_\varepsilon(x,n) =\sup_{y\in B_n(x,\varepsilon)}\sum_{k=0}^{n-1}\phi(f^k(y))$. In particular, $\Phi_0(x,n)=\sum_{k=0}^{n-1}\phi(f^k(y))$.

We identify $M \times \mathbb{N}$ with the space of finite orbit segments by identifying $(x, n)$ with $\{x, f(x),\cdots, f^{n-1}(x) \}$.
Given $\mathcal{C}\subset M\times \mathbb{N}$ and $n\in \mathbb{N}$ we write $\mathcal{C}_n:=\{x\in M: (x,n)\in \mathcal{C}\}$.  
Fixed $\varepsilon,\delta>0$ and $n\in \mathbb{N}$ we consider the {\em partition function}
$$
\Lambda(\mathcal{C},f, \phi,\delta, \varepsilon,n ):=\sup \left\{\sum_{x \in E} e^{\Phi_{\epsilon}(x,n)}: E \subset \mathcal{C}_n \text { is }(n, \delta) \text {-separated }\right\},$$
and when $\varepsilon=0$ we write   $\Lambda(\mathcal{C},f,\phi,\delta,n)$ instead of $\Lambda(\mathcal{C},f,\phi,\delta,0,n)$. 
Note that $\Lambda$ is monotonic in both
$\delta$ and $\varepsilon$, although in different directions: if $\delta_1 <\delta_2$ and  $\varepsilon_1<\varepsilon_2$ then

$$\Lambda(\mathcal{C},f, \phi, \delta_1, \varepsilon, t) \geq \Lambda(\mathcal{C}, f, \phi, \delta_2, \varepsilon, t) 
\text{ and }\Lambda(\mathcal{C}, f, \phi, \delta, \varepsilon_1, t) \leq \Lambda(\mathcal{C}, f,\phi, \delta, \varepsilon_2, t).$$

The pressure of $\phi$ on $\mathcal{C}$ at scale $\delta,\varepsilon$  is
$$
P(\mathcal{C},f, \phi, \delta, \varepsilon):=\limsup_{n \to \infty} \frac{1}{n} \log \Lambda_n(\mathcal{C},f, \phi,\delta, \varepsilon).
$$
Note also that the monotonicity of $\Lambda$ is naturally translated to $P$.
When $\varepsilon=0$ we simplify the notation and write $P(\mathcal{C},f,\phi,\delta)$ instead 
of $P(\mathcal{C},f,\phi,\delta,0)$ and the pressure of $\phi$ on $\mathcal{C}$ is
$$
P(\mathcal{C},f, \phi):=\lim_{\delta \to 0} P(\mathcal{C},f, \phi, \delta) .
$$
Given $C \subset M$, we define $P(C,f, \phi,\delta, \varepsilon):=P(C \times \mathbb{N}, f,\phi,\delta, \varepsilon)$; observe that $P(C, f, \phi )$ agrees with the usual notion of topological pressure in $C$. When $\phi=0$, we have the topology entropy of $\mathcal{C}$, that is,
$$
h_{top}(\mathcal{C},f, \delta ):=P(\mathcal{C},f,0, \delta) \quad \text {and} \quad h_{top}(\mathcal{C},f)=\lim_{\delta \to 0} h(\mathcal{C}, {f}, \delta) 
.$$
Finally, we define
$$P(f,\phi):=P(M,f,\phi) \quad\text{and}\quad h_{top}(f):=P(M,f),$$
which are the usual notions of pressure and entropy.

Write $\cM(M)$ for the set of the Borel probability measures of $M$, $\cM_f$ the set of $f$-invariant Borel probability measures and $\cM^e_f$ for the set of ergodic measures in $\cM_f$. The variational principle states that
$$
P(f,\phi)=\sup_{\mu \in \mathcal{M}_f}\left\{h_\mu(f)+\int \phi d \mu\right\}=\sup_{\mu \in \mathcal{M}_f^e}\left\{h_\mu(f)+\int \phi d \mu\right\} .
$$

A measure achieving the supremum is said an equilibrium state (abbrev. EE). When the potential is zero, the measure equilibrium state is called a measure of maximal entropy (abbrev. MME). Additionally, for $\mu \in \mathcal{M}_f$ we define 

$$P_\mu(f,\phi):=h_\mu(f)+\int\phi d\mu.$$

\subsubsection{Obstruction to expansivity}

We start defining the bi-infinite Bowen ball around $x \in {M}$ of size $\varepsilon>0$ as the set
$$
\Gamma_{\varepsilon}(x):=\left\{y \in M: d\left(f^k x, f^k y\right)<\varepsilon \text { for all } n \in \mathbb{Z}\right\} .
$$
If there exists $\varepsilon>0$ for which $\Gamma_{\varepsilon}(x)=\{x\}$ for all $x \in M$, we say that $f$ is expansive. 

When there exists $\varepsilon$ such that $h_{top}(\Gamma_\varepsilon(x))=0$ for every $x\in M$ we say that $f$ is entropy expansive, $h-$expansive for short.  

The set of non-expansive points at scale $\varepsilon$ is 
$$\operatorname{NE}(\varepsilon):=\left\{x \in X: \Gamma_{\varepsilon}(x) \neq\{x\}\right\}.$$ 
An $f$-invariant measure $\mu$ is almost expansive at scale $\varepsilon$ if $\mu(\mathrm{NE}(\varepsilon))=0$. 

\begin{definition}\label{def-obs-pressure}
    
Given a potential $\phi$, the pressure of obstructions to expansivity at scale $\varepsilon$ is
$$
\begin{aligned}
P_{\exp }^{\perp}(f,\phi, \varepsilon) & =\sup_{\mu \in \mathcal{M}^e(f)}\left\{h_\mu(f)+\int \phi d \mu: \mu(\mathrm{NE}(\varepsilon))>0\right\} \\
& =\sup_{\mu \in \mathcal{M}^e(f)}\left\{h_\mu(f)+\int \phi d \mu: \mu(\mathrm{NE}(\varepsilon))=1\right\}.
\end{aligned}
$$
\end{definition}

\subsubsection{ Weak specification}
The specification property plays a central role in the
work of Bowen \cite{Bowen71} and Climenhaga-Thompson \cite{climThom2016}. In \cite{climThom2016} a weak specification definition is introduced for a set of finite orbit segments.\begin{definition}
   We say that $\mathcal{G}\subset M\times \mathbb{N}$ has weak specification at scale $\delta>0$ if there exists $\tau\in \mathbb{N}$ such that  for every $(x_i,n_i)_{i=1}^k \subset \mathcal{G}$ there exists a point $y$ and a sequence of \textquotedblleft glueing times" $\tau_1,\cdots, \tau_{k-1}\subset \mathbb{N}$ with $\tau_i\leq \tau$ such that writing  $N_j=\sum_{i=1}^{j}n_i+\sum_{i=1}^{j-1}\tau_i$, and $N_0=\tau_0=0$ we have

    \begin{equation*}
        d_{n_j}(f^{n_{j-1}+\tau_{j-1}}(y),x_j)<\delta \mbox{ for every }1\leq j\leq k 
    \end{equation*}.
 \end{definition}

 The specification property implies the existence of a point $y$ whose orbit closely shadows the orbit of $x_1$ for a certain period time $n_1$. It then transitions to shadow the orbit of $x_2$ for another period  time $n_2$, with bounded gaps no greater than $\tau$ between each transition.
 
\subsubsection{Bowen's property}
The bounded distortion property also referred to as Bowen's property, was initially introduced by Bowen in  \cite{Bowen71}.  

\begin{definition}
Given $\mathcal{G} \subset M \times \mathbb{N}$, a potential $\phi$ has the Bowen property on $\mathcal{G}$ at scale $\varepsilon>0$ if there exists $K>0$ so that

$$\sup \left\{\left|\Phi_0(x,n)-\Phi_0(y,n)\right|:(x, n) \in \mathcal{G}, y \in B_n(x, \varepsilon)\right\} \leq K.$$
\end{definition}

\subsubsection{Dynamic Decompositions}\label{decompositions}
The most important observation in \cite{climThom2016} is that unique equilibrium state can be obtained if the specification and Bowen property established in \cite{Bowen74} are satisfied only on a significant set of orbit segments rather than  the whole variety.

\begin{definition}
A decomposition for $(M, f)$ consists of three collections $\mathcal{P}, \mathcal{G}, \mathcal{S} \subset M \times(\mathbb{N} \cup\{0\})$ and three functions $\hat{p}, \hat{g}, \hat{s}: M \times \mathbb{N} \to \mathbb{N} \cup\{0\}$ such that for every $(x, n) \in$ $M\times \mathbb{N}$, the values $\hat{p}=\hat{p}(x, n), \hat{g}=\hat{g}(x, n)$, and $\hat{s}=\hat{s}(x, n)$ satisfy $n=\hat{p}+\hat{g}+
\hat{s}$, and
$$
(x, \hat{p}) \in \mathcal{P}, \quad\left(f^{\hat{p}} (x), \hat{g}\right) \in \mathcal{G}, \quad\left(f^{\hat{p}+\hat{g}}(x), \hat{s}\right) \in \mathcal{S} .
$$
\end{definition}

Note that the symbol $(x, 0)$ denotes the empty set, and the functions $\hat{p}, \hat{g}, \hat{s}$ are permitted to take the value zero.

The main result  in \cite{climThom2016} establishes the existence and uniqueness of equilibrium states for systems with the Bowen property, weak specification, and a minor obstruction for expansiveness on a specific set of \textquotedblleft good orbits." 
However, this criterion is not applicable in our specific scenario due to its rigorous assumptions on the specification. To overcome this challenge, we will use the improved of Climenhaga-Thompsons criterion by Pacifico et al established in \cite{PYY2022}.

\begin{theorem}\cite[Theorem A]{PYY2022}\label{criterioCT}
Let $f: M \to M$  \st{Lipschitz} be a homeomorphism defined in a  compact metric space $M$ and $\phi: M \to \bR$  a continuous potential function. Suppose  there are $\varepsilon,\delta$  and $\varepsilon>2000\,\delta$ \st{where $L_f>0$ is a constant that depends of $f$. Suppose} such that $P_{\exp }^{\perp}(f,\phi, \varepsilon)<P(f,\phi)$ and $\cD \subset M\times \bN$ which 
 admits a decomposition $(\mathcal{P}, \mathcal{G}, \mathcal{S})$ with the following properties:
\begin{enumerate}
    \item $\mathcal{G}$ has (W)-specification at scale $\delta$; 
    \item $\phi$ has the Bowen property at scale $\varepsilon$ on $\mathcal{G}$;
    \item $P(\mathcal{P} \cup \mathcal{S},g,\phi,\delta,\varepsilon)<P(f,\phi)$.
\end{enumerate}

Then there is a unique equilibrium state for $(M,f,\phi)$.
\end{theorem}

\subsection{Partial hyperbolic systems}\label{ss-PH}
The concept of  partially hyperbolic systems is a natural generalization  of uniform hyperbolicity, and the research in this area dates back to the early 1970s, see, for instance \cite{HPS77}.

\begin{definition}\label{def-partialh}
     A diffeomorphism $f$ defined on a  Riemannian compact manifold $M$ is  partially hyperbolic (abbrev. PH) if it  admits a non-trivial $Df$- invariant splitting of the tangent bundle $TM=E^s_f\oplus E^c_f\oplus E^u_f$, such that, all unit vectors $v^{\sigma}\in E^\sigma_f(x)$ ($\sigma=s,c,u$) with $x\in M$ satisfy

\begin{equation*}
    \|Df_xv^s\|<\|Df_xv^c\|<\|Df_xv^u\|
\end{equation*}
for some appropriate Riemannian metric; f must also satisfy that $0<\|Df_{|E^s}\|<\xi<1$ and $0<\|Df^{-1}_{|E^u}\|<\xi<1$.

\end{definition}
Throughout this paper, we will work with  partially hyperbolic diffeomorphisms with dimension central one, that is, $\dim\, (E^c)=1$. 

\subsubsection{Minimal foliations}
One of the fundamental results in the theory of partially hyperbolic dynamical systems is the existence of foliations $\cF^\sigma_f$,  tangent to the stable and unstable distributions $E^\sigma_f$ of  $f$, ($\sigma=s,u$). Those foliations are called unstable and stable foliations respectively.  Specifically, for any $x\in M$, there exists a leaf of $\cF^\sigma_f(x)$  containing $x$ and it corresponds to the classical unstable or stable manifold  $W^\sigma(x,f)$, $(\sigma=s,u)$, as shown in \cite{HPS77, BP73}
\begin{definition}
Consider a partially hyperbolic diffeomorphism $f: M \to M$. The foliation $\cF^\sigma_f$ is minimal if $W^\sigma(x,f)$, for all $x\in M$, is dense in $M$, $(\sigma=s,u)$.     
\end{definition}

\begin{definition}\label{defi-epsilon-minimality}
    Let $f:M\to M$ be a partially hyperbolic diffeomorphism. The  unstable foliation $\mathcal{F}^u_f$ of $f$ is $\varepsilon-$minimal if  there exists $R> 0$ such that if $D$ is a disk contained in an  unstable leaf of $\cF^u_f$  with an internal radius larger than $R$ then $D$ is $\varepsilon$-dense em $M$. 
\end{definition}

It is well known that if $f:M\to M$ is a partially hyperbolic diffeomorphism whose unstable foliation is minimal, then, for every $\varepsilon> 0$ there exists  a  $C^1$ neighborhood $\mathcal{U}$ of $f$ such that for any $g\in \mathcal{U}$, the unstable foliation is $\varepsilon$-minimal. 

\subsubsection{ Existence of equilibrium states}
In general, it can be difficult to prove that a partially hyperbolic system has equilibrium states.
   However, in some cases, we can solve this task using the concept of
   pressure of $f$, $P_{\mu}(f,\phi) $,  defined as 
 $$P_\mu(f, \phi):  \mathcal{M}_f \to \bR, \,\, \mu \mapsto P_{\mu}(f, \phi),\,\,
 \mbox{where $\phi$ is a continuous observable.}
 $$

Since $\cM_f$ is a compact set, if the pressure  is upper semicontinuous then
it achieves a maximal value and so $(f, \phi)$ has an equilibrium state. 
 
 For homeomorphisms, Bowen showed \cite{Bowen72} that the metric entropy
 is upper semicontinuous whenever $(M, f)$ is a $h$-expansive system. 
In particular, the pressure $P_\mu(f,\phi)$ is also upper semicontinuous  and then
$(f, \phi)$ admits an equilibrium state for any continuous potential $\phi$.
Hence,  since hyperbolic systems with one-dimensional center bundle are $h$-expansive,
 \cite{CY05, DFPV12, LVY13},  they have equilibrium states.

\subsection{Unstable entropy }

In this section, we recall the notions of unstable $h^u_\mu(f)$ and stable $h^s_\mu(f)$  metric entropy for a partial hyperbolic diffeomorphism $f$ introduced in \cite{yang2016, HHHY2017}.

\subsubsection{Unstable metric entropy}
Here we follow closely  \cite{HHHY2017}. 

Consider a partially hyperbolic diffeomorphism $f$  such that $\operatorname{dim}(E^u_f)\geq 1$. Let $\alpha$ be a  partition of $M$. We denote $\alpha(x)$ as the element of $\alpha$ containing the point $x$. If we have two  partitions $\alpha$ and $\beta$ such that $\alpha(x) \subset \beta(x)$ for all $x \in M$, we can write $\alpha \geq \beta$ or $\beta \leq \alpha$. if a partition $\xi$ satisfies $f^{-1}(\xi) \geq \xi$ we say that the partition is increasing.

For a measurable partition $\beta$, we use the notation $\beta_n^m = \bigvee\limits_{i=m}^n f^{-i}(\beta)$. In particular, $$\beta_{n-1}^0 = \bigvee\limits_{i=0}^{n-1} f^{-i}(\beta).$$

Take $\epsilon_0 > 0$  small and let $\cP = \cP_{\varepsilon_0}$ represent the set of finite measurable partitions of $M$ where each element has a diameter smaller than or equal to $\epsilon_0$. For each $\beta \in \cP$, we can define a finer partition $\eta$ such that $\eta(x) = \beta(x) \cap W^{u}_{\operatorname{loc}}(x)$ for every $x \in M$. Here, $W^{u}_{\operatorname{loc}}(x)$ represents the local unstable manifold at $x$ that has a size greater than the diameter $\epsilon_0$ of $\beta$. 
Note that $\eta$ is a measurable partition that satisfies $\eta \geq \beta$. We denote the set of such partitions as $\cP_{u} = \cP_{u,\epsilon_0}$.

A measurable partition $\xi$ of $M$ is subordinate to the unstable manifold of $f$ with respect to a measure $\mu$ if, for $\mu$-almost every $x$, $\xi(x)$ is a subset of $W^{u}(x)$ and contains an open neighborhood of $x$ within $W^{u}(x)$. 
If $\alpha \in \cP$ and $\mu(\partial\alpha) = 0$, where $\partial\alpha := \bigcup\limits_{A \in \alpha} \partial A$, then the corresponding $\eta$ given by $\eta(x) = \alpha(x) \cap W^{u}_{\operatorname{loc}}(x)$ is a partition that is subordinate to the unstable manifold of $f$.

Let us recall that, given measurable partition $\eta$ of a measure space $X$ and a probability measure $\nu$ defined on $X$, the canonical system of conditional measures for $\nu$ and $\eta$ is a collection of probability measures ${\nu_x^\eta : x \in X}$ satisfying $\nu_x^\eta(\eta(x)) = 1$. These measures have the property that for any measurable set $B \subset X$, the function $x \mapsto \nu_x^\eta(B)$ is measurable and the integral equation
$$
\nu(B)=\int_X \nu_x^\eta(B) d \nu(x)
$$
(See e.g. \cite{RO49} for reference.)

Remind that the information function of $\alpha \in \mathcal{P}$ is defined as
$$
I_\mu(\alpha)(x):=-\log \mu(\alpha(x))
$$
and the entropy of the partition $\alpha$ as
$$
H_\mu(\alpha):=\int_M I_\mu(\alpha)(x) d \mu(x)=-\int_M \log \mu(\alpha(x)) d \mu(x) .
$$
The conditional information function of $\alpha \in \mathcal{P}$ with respect to a measurable partition $\eta$ of $M$ is given by
$$
I_\mu(\alpha \mid \eta)(x):=-\log \mu_x^\eta(\alpha(x))
$$
Then the conditional entropy of $\alpha$ with respect to $\eta$ is defined as
$$
H_\mu(\alpha \mid \eta):=\int_M I_\mu(\alpha \mid \eta)(x) d \mu(x)=-\int_M \log \mu_x^\eta(\alpha(x)) d \mu(x) .
$$
We now introduce the notion of unstable metric entropy presented in \cite{HHHY2017} which is similar to the classical metric entropy but incorporates the use of a conditional partition $\eta$ to exclude the influence of central directions.
\begin{definition}
    The conditional entropy of $f$ with respect to a measurable partition $\alpha$ given $\eta \in \mathcal{P}^u$ is defined as
$$
h_\mu(f, \alpha \mid \eta)=\limsup_{n \rightarrow \infty} \frac{1}{n} H_\mu\left(\alpha_0^{n-1} \mid \eta\right) .
$$
The conditional entropy of $f$ given $\eta \in \mathcal{P}^u$ is defined as
$$
h_\mu(f \mid \eta)=\sup_{\alpha \in \mathcal{P}} h_\mu(f, \alpha \mid \eta)
$$
and the unstable metric entropy of $f$ is defined as
$$
h_\mu^u(f)=\sup_{\eta \in \mathcal{P}^u} h_\mu(f \mid \eta)
.$$
\end{definition}

It's possible to prove that $h_\mu(f \mid \eta)$ is independent of $\eta$, as long as it is in $\mathcal{P}^u$. Hence, we  have $h_\mu^u(f)=h_\mu(f \mid \eta)$ for any $\eta \in \mathcal{P}^u$.

If the dimension of the stable bundle $E^s_f$ is greater than or equal to 1, we can define the metric stable entropy for any $\mu \in \mathcal{M}_f$ as $h^s_\mu(f) := h^u_\mu(f^{-1})$. 

\subsubsection{ Unstable topological entropy and the variational principle}
Now we start to define the unstable topological entropy, introduced in \cite{SX10}.
We denote by $d^u$ the metric induced by the Riemannian structure on the unstable manifold and let $d_n^u(x, y)=\max _{0 \leq j \leq n-1} d^u\left(f^j(x), f^j(y)\right)$. Let $W^u(x, \delta)$ be the open ball inside $W^u(x)$ centred at $x$ of radius $\delta$ with respect to the metric $d^u$. Let $N^u(f, \epsilon, n, x, \delta)$ be the maximal number of points in $\overline{W^u(x, \delta)}$ with pairwise $d_n^u$-distances at least $\varepsilon$. 
We call such a set an $(n, \varepsilon)$ u-separated set of $\overline{W^u(x, \delta)}$.

\begin{definition}\label{def-u-s-entropia}
The unstable topological entropy of $f$ on $M$ is defined by
$$
h_{\text {top }}^u(f)=\lim_ {\delta \rightarrow 0} \sup _{x \in M} h_{\text {top }}^u\left(f, \overline{W^u(x, \delta)}\right),
$$
where
$$
h_{\text {top }}^u\left(f, \overline{W^u(x, \delta)}\right)=\lim _{\epsilon \rightarrow 0} \limsup _{n \rightarrow \infty} \frac{1}{n} \log N^u(f, \epsilon, n, x, \delta) .
$$    
\end{definition}

We can also define unstable topological entropy using $(n, \epsilon)$ u-spanning sets or open covers to get equivalent definitions. 
Analogously, if the dimension of the stable bundle $E^s_f$ is greater than or equal to 1, we can define the stable entropy for any $\mu \in \mathcal{M}_f$ as $h^s(f) := h^u_\mu(f^{-1})$.

As in the case of the usual definition of entropy, we can relate the metric unstable entropy with the unstable topological entropy by a variational principle. Indeed,  
\cite [Theorem D]{HHHY2017} estates that if $f : M \rightarrow M$ is a $C^{1}$-partially hyperbolic diffeomorphism then it holds
$$h^u_{top}(f) = \sup\{h_{\mathrm{u}}^{\mu}(f) : \mu \in \mathcal{M}_f\}
\,\,\,\mbox{and}\,\,\, h^u_{top}(f) = \sup\{h^u_{\nu}(f) : \nu \in \mathcal{M}^e_{f}\}.$$

{An alternative definition of topological unstable entropy can be derived  considering the unstable volume growth given by Hua, Saghin, and Xia (\cite{HYS08}), reminiscent from the works by Yomdin and Newhouse (\cite{Y87, N89}). 
In \cite[Theorem C]{HHHY2017}  shows that the unstable topological entropy as defined here coincides with the unstable volume growth.}
\section{Consequences of $h^u(f)-h^s(f)>0$}\label{s-consequences} 
Let $\Diff^{1 +}(M)$ be the set of $C^{1 +}$ diffeomorphisms defined on $M$.
From now on, $f$ will be a $C^{1 +}$ partially hyperbolic diffeomorfism with $1$-dimensional central direction with non trivial stable and unstable bundles, that is, $\operatorname{dim}(E^{\sigma})>0 (\sigma=s,u),$ and $\phi$  a H\"older continuous potential.

If $\mu$ is an ergodic  $f$-invariant measure we set $\lambda^c(\mu,f)$ for the Lyapunov exponent of $f$ in the central direction. An $f$-invariant ergodic measure $\mu$ is  {\em hyperbolic} if the central Lyapunov exponent satisfies $\lambda_c(f,\mu)\neq 0$.

Given $f\in \Diff^{1+}(M)$ and $\mu$ ergodic
the Ledrappier and Young formula establish that 

\begin{equation}\label{L-Yformula}
    h_\mu(f)\leq h_\mu^u(f)+ \sum_{\lambda_c(f,\mu)>0}\lambda_c(f,\mu) 
\end{equation}

This formula was proved in the case of $C^2$--diffeomorphisms  by Ledrappier and Young in \cite[Theorem C]{LYI85,LYII85} and it was proved for $C^{1+}$--diffeimorphisms  in \cite[Theorem 5.2]{Br22}. A directly consequence of this formula is that  if the central Lyapunov exponent of $\mu$ is non-positive then 
\begin{equation}\label{u-entropy-metric=entropy-metric}
    h_\mu(f)= h_\mu^u(f). 
\end{equation}

\begin{lemma}\label{mme-implies-center-lyapunov-exponent<0}
 Let $\mu$ be an ergodic equilibrium state for $(M,f,\phi)$. If $h^u(f)-h^s(f)>\sup \phi-\inf \phi\geq 0$, then $\lambda^c(\mu,f)<0$. 
\end{lemma}
\begin{proof}
Suppose that $\mu$ is an ergodic equilibrium state for $f$ with a non-negative center Lyapunov exponent $\lambda^c(\mu,f)\geq 0$. By  (\ref{u-entropy-metric=entropy-metric}) we have
$$P_\mu(f,\phi)=h_\mu(f)+\int \phi d\mu=h^s_\mu(f)+\int \phi d \mu \leq h^s(f)+\sup \phi.$$

By hypothesis, $h^u(f)+\inf \phi>h^s(f)+\sup\phi $ and $h^u(f)\leq h_{top}(f)$ and thus we get
$P_\mu(f,\phi)<h^u(f)+\inf \phi\leq P(f,\phi)$, which contradicts that $\mu$ is an ergodic equilibrium state. This finishes the proof.
\end{proof}
\begin{definition}
        The unstable pressure is defined as 
        $$P^u(f,\phi)=\sup\{h_\mu^u(f) + \int\phi d\mu, \mu \in \mathcal{M}_f\}.$$
\end{definition}

\begin{remark}\label{h_{top}(f)=humu(f)} 
  By equation (\ref{u-entropy-metric=entropy-metric}), if $\mu$ is an ergodic measure with $\lambda_c(\mu, f)<0$, we obtain $P_\mu(f,\phi)=P^u_\mu(f,\phi)=h^u_{\mu}(f)+\int\phi d\mu$. 
  If in addition $\mu$ is a ergodic EE and $\phi$ satisfies $h^u(f)-h^s(f)>\sup \phi-\inf \phi\geq 0$, Lemma
  \ref{mme-implies-center-lyapunov-exponent<0} implies that $\lambda_c(\mu)<0$ and
  thus we obtain    $P(f,\phi)=h^u_\mu(f)+\int \phi d\mu= P^u(f,\phi)$.
\end{remark}

In the remainder  of this section, we prove  that the property 
\begin{equation}\label{gap-condition}
h^{u}(f)>\sup \phi\geq \inf \phi> h^s(f)
\end{equation}

 persists for diffeomorphisms $g \in  \operatorname{Diff}^{1+}(M)$ in an $C^1$ open neighborhood of $f$.

 For this, we proceed as follows.

\begin{lemma}\label{unstable-pressure-is-continues}
 The map $g \mapsto P^u(g,\phi)$ with $g\in  \operatorname{Diff}^{1+}(M)$ is a continuous function in $f$ with the $C^1$ topology whenever \eqref{gap-condition} holds.
\end{lemma}

\begin{proof} We start proving the following result guarantying that equilibrium states has
positive entropy, allowing to approximate the topological pressure  by  Katok's horseshoe, see \cite{K80, KH95}:

\begin{claim}\label{claim1}
If $\mu$ is an ergodic equilibrium state for $(f,\phi)$ then  $h_\mu(f)> 0$. 
\end{claim}

To achieve this, it is enough to verify that an ergodic maximal entropy measure has higher pressure than a measure with zero entropy. For this, we do as follows.

Let $\nu,\mu\in \mathcal{M}_f$  such that $h_\nu(f)=0$ and $\mu$ is an ergodic measure of maximal entropy. Thus $P_\nu(f,\phi) =\int \phi d\nu\leq \sup \phi$ and by hypothesis we get
\begin{equation}\label{e-pressao1}
P_\nu(f,\phi)\leq \sup \phi< h^u(f).
\end{equation}

By  Remark \ref{h_{top}(f)=humu(f)} when  the potential $\phi$ is identically zero,  $h^u(f)=h_{top}(f)$. 
Since $\mu$ is a maximal entropy measure we get $h^u(f)=h_{\mu}(f)$. Thus, using  (\ref{e-pressao1}) we get
\begin{equation*}
P_\mu(f,\phi)=h_\mu(f) + \int \phi d\mu \geq h_\mu(f) + \inf(\phi) \geq h_\mu(f)=h^u(f) > P_\nu(f,\phi).
\end{equation*}
This finishes the proof of  the Claim \ref{claim1}. \hspace{3cm} $\square$

Returning to the proof of Lemma \ref{unstable-pressure-is-continues},
 we will prove that the unstable pressure of $f$, $P^u(f)$, is continuous at $f$ when we restrict 
 it to  $C^{1+}$ diffeomorphisms.
 For this, since the map $g \mapsto P^u(g,\phi)$ is $C^1$ upper semicontinuous at $f$ \cite[Theorem A]{yang2016},
 it is enough to show that $P^u$ is lower semicontinuous at $f$.

 Let $\varepsilon>0$ and suppose that $\mu$ is an ergodic equilibrium state.
 By Lemma \ref{mme-implies-center-lyapunov-exponent<0} and  Remark \ref{h_{top}(f)=humu(f)}, we have that $P(f,\phi)=h^u_\mu(f)+\int \phi d\mu=P^u(f,\phi)$ and  $\lambda^c(\mu,f)<0$.

Since  $\mu$ is a hyperbolic measure with $h_\mu(f) > 0$ (see Claim \ref{claim1}), by classical results from Katok \cite{K80, KH95} (see also \cite[Theorem 1]{Gel16}), we conclude that there exists a hyperbolic set $\Lambda_\varepsilon \subset M$ such that

\begin{equation}\label{aproximate-by-hyperboly-set}
  P^u(f,\phi)-\frac{\varepsilon}{2}\leq P(f_{|\Lambda_\varepsilon},\phi_{|\Lambda}). 
\end{equation}

As $\Lambda_\varepsilon$ is a hyperbolic set for $f$, we have that $(f_{|\Lambda_\varepsilon},\Lambda_\varepsilon)$ is an expansive map and therefore, there exists an ergodic measure $\mu_\varepsilon$ such that $P_{\mu_\varepsilon}(f,\phi)=P(f_{|\Lambda_\varepsilon}\phi_{|\Lambda_\varepsilon}).$ Thus
$$
 P^u(f,\phi)-\frac{\varepsilon}{2}\leq 
 P_{\mu_\varepsilon}(f,\phi).
 $$

Since hyperbolic systems are structurally stable,  there exists 
a $C^1$ neighborhood $\mathcal{U}(f)$ of $f$ such that if $g\in \mathcal{U}(f)$, there exists an $g-$invariant hyperbolic set  $\Lambda_g\subset M$ such that $g_{|\Lambda_g}$ and $f_{|\Lambda_\varepsilon}$ are topologically conjugate by a homeomorphism $h_g:M \to M$ 
satisfying 
\begin{equation}
    \|h_g-I\|<\delta,
\end{equation}
where $\delta>0$ is such that if $\|I-h\|<\delta$ then 
\begin{equation}\label{close-conjugated}
    \|\phi-\phi\circ h_g\|<\frac{\varepsilon}{2}.
\end{equation}

Since $h_g$ conjugates $f$ and $g$, for every $g \in \mathcal{V}$, we obtain 

\begin{equation}\label{pressure-iqual}
P(f_{|\Lambda_\varepsilon}\phi_{|\Lambda_\varepsilon})=P(g_{|\Lambda_g},
\phi\circ h_{|\Lambda_g}),
\end{equation}
and by  (\ref{close-conjugated}), we obtain

\begin{equation}\label{close pressure}
P(g_{|\Lambda_g},\phi\circ h_{|\Lambda_g})-\frac{\varepsilon}{2}\leq P(g_{|\Lambda_g},\phi_{|\Lambda_g}).
\end{equation}

Now, since $g_{|\Lambda_g}$ is conjugated with $f_{|\Lambda_{\varepsilon}}$ and $\Lambda_g$ is a hyperbolic set for $g$, we obtain that  $(g_{|\Lambda_g},\phi_{|\Lambda_g)}$ has an ergodic equilibrium state $\mu_g$ with $\lambda_c(\mu_g,g)<0$. 
Thus, by  (\ref{u-entropy-metric=entropy-metric}), 
we have that $P(g_{|\Lambda_g},\phi_{|\Lambda_g)}=P^u_{\mu_g}(g,\phi)\leq P^u(g,\phi)$ for every $g\in  \operatorname{Diff}^{1+}(M)$. 
Therefore, using  (\ref{pressure-iqual}) and (\ref{close pressure}) we 
obtain

\begin{equation*}
    P(f_{|\Lambda_\varepsilon},\phi)-\frac{\varepsilon}{2}\leq P^u(g,\phi),
\end{equation*}
and by  (\ref{aproximate-by-hyperboly-set}) we have that
$$P^u(f,\phi)-\varepsilon\leq P^u(g,\phi).$$
Therefore, $P^u(f,\phi)-\varepsilon\leq P^u(g,\phi)$ for every $g\in \mathcal{U} \cap  \operatorname{Diff}^{1+}(M)$ which implies that $g\mapsto P^u(g,\phi)$ is lower semicontinuous.
\end{proof}

    \begin{corollary}\label{hu>hs aberto} If there are positive numbers  $a$ and $b$ such that $h^u(f)> a > b>  h^s(f)$ then  there exists a $C^1$-neighborhood $\widetilde{\mathcal{V}}$ of $f$ such that for every  $g\in \widetilde{\mathcal{V}}\cap  \operatorname{Diff}^{1+}(M)$, $h^u(g)>a>b>h^s(g)$. 
\end{corollary}

\begin{proof} Note that in the case when \st{in the case} \eqref{gap-condition} holds 
 it is sufficient to take $a=\sup \phi$ and $b=\inf \phi$.

Assume there are such positive numbers numbers $a$ and $b$ satisfying 
$h^u(f)> a > b>  h^s(f)$. Then, $h^u(f) - h^s(f)> 0$ and we can apply  Lemma \ref{unstable-pressure-is-continues}, obtaining   that 
 the unstable entropy $h^u:  \operatorname{Diff}^{1+}(M) \to \bR$ is continuous at $f$ with the $C^1$ topology.
 Hence, given $\varepsilon > 0$, there is a $C^1$-neighborhood $\widetilde{\mathcal{V}}
$ of $f$ such that if $g \in \widetilde{\mathcal{V}} \cap  \operatorname{Diff}^{1+}(M)$ then 
 $|h^u(f) - h^u(g)| < \varepsilon$.
 
 Since the stable entropy $h^s: \operatorname{Diff}^{1+}(M) \to \bR$ is upper semicontinuous, we can take a neighborhood 
 $\widetilde{\cU} \subset  \operatorname{Diff}^{1}(M)$ of $f$ such that if  $g\in \widetilde{\cU}$ then
  $h^s(g) - h^s(f) < \varepsilon$.
  
Now, let $0<\varepsilon<\frac{\min\{h^u-a,b-h^s\}}{2}$ and $\cV=\widetilde{\mathcal{V}}
  \cap \widetilde{\cU}$. 
  Then, for all $g \in \cV$ we have $h^u(g)> a>b > h^s(g)$, finishing the proof.
\end{proof}

\begin{corollary}\label{pressure-continua}
The map $P(\cdot, \phi):  \operatorname{Diff}^{1+}(M) \to \bR$
is continuous at $f$ in the $C^1$ topology whenever  \eqref{gap-condition} holds.
\end{corollary}

\begin{proof}
 Without loss of generality, assume suppose that  $h^u(f)> \sup \phi>\inf \phi> h^s(f)$. Let $\varepsilon>0$.
    Lemma \ref{unstable-pressure-is-continues} and Corollary \ref{hu>hs aberto} imply that there exists a $C^1$ neighborhood of $f$ so that for every $g\in \mathcal{U}$
    $|P^u(f,\phi)-P^u(g,\phi)|<\varepsilon$ and $h^u(g)> \sup \phi>\inf \phi> h^s(g)$. 
    
    Therefore, by the Remark \ref{h_{top}(f)=humu(f)}, $P^u(g,\phi)=P(g,\phi)$ and thus
    $$|P(f,\phi)-P(g,\phi)|<\varepsilon,$$
    for every $g\in \mathcal{U},$ ending the proof.
\end{proof}

\section{ Proof of Theorem \ref{teo-main}}\label{s-prova-teoA}

\subsection{Choosing the decomposition}\label{choosing descomposition}

In this section, we follow closely \cite{climThom2021}.
The aim is to provide a large collection of orbit segments  with a decomposition 
$(\cP_g, \cG_g, \cS_g)$ as in Section \ref{decompositions} in such a way that $\cG_g$ captures all the "hyperbolicity" of every $g$ close to $f$ in the $C^1$ topology,  to overcome  lack of hyperbolicity outside $\cG$. 

 Let $\widetilde{\cV}$ be a $C^1$ neighborhood of $f$ as in Corollary \ref{hu>hs aberto}.  We can assume, without loss of generality, for $g\in \widetilde{\cV}$ it holds $\|Dg_{E^s}\|,\| Dg^{-1}_{E^u}\|<\xi$, where $\xi$ is given at Definition \ref{def-partialh}. 

For every $g\in \widetilde{\mathcal{V}}\cap  \operatorname{Diff}^{1+}(M)$ we set $\varphi^c(x):=\log\|Dg_{|E^c(x)}\|$. 
If $\mu \in \mathcal{M}_g$ we set  $\lambda_c(\mu,g):=\int \varphi^c d\mu.$ 
When $\mu$ is ergodic and the central direction of a $C^{1 +}$ partially hyperbolic diffeomorphism is $1$-dimensional, $\lambda_c(\mu,g)$ coincides with the central Lyapunov exponent of $g$.
Let $ P^+(g,\phi)$ and $ P^-(g,\phi)$ be defined as below:
\begin{equation*}
 P^+(g,\phi)=\sup \{P_\mu(g,\phi): \mu\in M^e_g,\lambda_c(\mu,g)\geq 0\}
\end{equation*}
and
\begin{equation*}
    P^-(g,\phi)=\sup \{P_\mu(g,\phi) : \mu\in M^e_g,\lambda_c(\mu,g)\leq 0\}.
\end{equation*}
 By the Corollary \ref{hu>hs aberto} and the Lemma \ref{mme-implies-center-lyapunov-exponent<0} there is no ergodic equilibrium state $\mu$ with  $\lambda_c(\mu,g)\geq 0$ and the pressure is upper semicontinuous, we have that 
 $$P^+(g,\phi)<P^-(g,\phi) \text{ for every } g\in \widetilde{\cV}\cap  \operatorname{Diff}^{1+}(M),$$
 and hence by the ergodic decomposition theorem, we have
$$\sup \{P_\mu(g,\phi): \mu\in \cM_g,\lambda_c(\mu,g)\geq 0\}<P(g,\phi) \text{ for every } g\in \widetilde{\cV} \cap  \operatorname{Diff}^{1+}(M).$$

%
%
%

\begin{lemma}\label{bad-UH-small-pressure}
    There exists $r>0$ and a $C^1$ neighborhood ${\cV}'\subset \widetilde{\cV}$ such that 
    $$\sup \{P_\mu(g,\phi): \mu\in \mathcal{M}_g,\lambda_c(\mu,g)\geq -r\}<P(g,\phi),
    \mbox{ for every }  g\in \cV'\cap   \operatorname{Diff}^{1+}(M).$$
\end{lemma}

\begin{proof}
  Suppose, by contradiction,  that there exist sequences $g_n\in \widetilde{\cV} \cap   \operatorname{Diff}^{1+}(M)$ and  $r_n>0$ with $r_n\to 0$ and $g_n \to f$ when $n\to \infty$   such that 
  $$\sup \{P_\mu(g_n,\phi): \mu\in \mathcal{M}_{g_n},\lambda_c({\mu},g_n)\geq -r_n\}=P(g_n,\phi).$$ 

     Since $\mu\mapsto \lambda_c(\mu,g)$ is a continuous map in the weak* topology, there is an invariant measure $\nu_n\in \mathcal{M}_g$ such that $P_{\nu_n}(g_n,\phi)= P(g_n,\phi)$ and $\lambda(\nu_n,g_n)\geq -r_n$.

     Since $\mathcal{M}(M)$ is compact, we can suppose that $\nu_n$ converges to $\nu$ and so $\nu$ is a  $f-$invariant and satisfies $\lambda_c(f,\nu)\geq 0$. Therefore, by Corollary \ref{pressure-continua}  and   \cite[Theorem A]{yang2016} we obtain
     $$
     P(f,\phi)=\lim_{n} P(g_n,\phi)=\lim_{n} P_{\nu_n}^u(g_n)\leq P^u_{\nu}(f)\leq P^u_\nu(f,\phi),
     $$
   implying that $\nu$ is an equilibrium state with $\lambda_c(f,v)\geq 0$, which is a contradiction by  Lemma \ref{mme-implies-center-lyapunov-exponent<0}.
\end{proof}

The proof of the corollary below is similar to the proof of Lemma \ref{bad-UH-small-pressure}
and so it will omitted.

\begin{corollary}\label{existence of a and V}
    Consider $r>0$ and ${\cV}'\subset   \operatorname{Diff}^{1}(M)$ as in  Lemma \ref{bad-UH-small-pressure}. There exist $a>0$ and a $C^1$ neighborhood $\mathcal{V}\subset \mathcal{V}'$ of $f$ such that, for every $g\in {\cV}\cap   \operatorname{Diff}^{1+}(M)$
\begin{equation}\label{hPP<Pto}
    \sup \{P_\mu(g,\phi): \mu\in \mathcal{M}_g,\lambda_c(\mu,g)\geq -r\}<a<P(g,\phi).
\end{equation}
\end{corollary}
%
%

Fix  $r>0$, $a>0$ and $\mathcal{V}\subset   \operatorname{Diff}^{1}(M)$ as in Corollary \ref{existence of a and V}.
Next we   define the decomposition $(\cP_g, \cG_g, \cS_g)$ for each
 $g \in \mathcal{V} \cap   \operatorname{Diff}^{1+}(M)$ as follows. 
 We put $\cS_g=\emptyset$, and define $\cP_g$ and $\cG_g$ as 

 \begin{equation}\label{def-Pg}
     \cP_g:=\{(x,n)\in M\times \mathbb{N}:S_n\varphi^c(x)\geq -rn\}
 \end{equation}
 and
 \begin{equation*}
     \cG_g:=\{(x,n)\in M\times \mathbb{N}: S_n\varphi^c(x)< -rn, \, \forall \, 0\leq j \leq n\}, \,\,\, S_n\varphi^c(x)=\sum_{k=0}^{n-1}\varphi^c(g^k(x)).
 \end{equation*}
 
 Take an arbitrary orbit segment $(x,n)\in M\times \mathbb{N}$ and let $\hat{p}=\hat{p}(x,n)$ be the maximal integer with the property that $(x,\hat{p})\in \mathcal{P}_g$ and $\hat{g}=\hat{g}(x,n)=n-\hat{p}$.

\subsection{ Specification on $\cG_g$}

To prove that each diffeomorphism $g$ close to
$f$ has specification in $\cG_g$ we will use the follow theorem.

\begin{theorem}\label{Manifolds for hyperbolic times}
    There exist a $C^1$ neighborhood $\cV$ of $f$, positive constants $R_{cs} = R_{cs}(f)$, $C = C(f)$, and $\tilde{\varepsilon} > 0$ such that if $g \in \cV$ and $(x,n) \in \cG_g$, the following holds:
    \begin{itemize}
        \item[i)] There exists a $C^1$ embedded disk $N^{cs}_{R_{cs}}(g,x)$ of dimension $\operatorname{dim}(E^{s}) + 1$ and radius $R_{cs} > 0$, centered at $x$, such that
        $$T_x N^{cs}_{R_{cs}}(g,x) = E^{cs}_g(x);$$
        \item[ii)] $N^{cs}_{R_{cs}}(g,x)$ depends continuously on both $x$ and $g$ in the $C^1$ topology;
        \item[iii)] if $y \in N^{cs}_{R_{cs}}(g,x)$, then 
        $$d(f^k(x), f^k(y))< Ce^{-kr/2}d(x,y) \mbox{ for every } 0\leq k\leq  n$$
        \item[iv)] If $y\in M$ with $d(x,y)<\hat{\varepsilon}$  then
        $$N^{cs}_{R_{cs}}(g,x)\cap W^u_{loc}(y)=\{z\},$$
        where $z\in M$ and $W^u_{loc}$ is the connected component contenning $x$ of $B(x,loc)\cap W^u(x)$ and $loc$ is a positive number. 
        \end{itemize}
\end{theorem}

The proof of Theorem \ref{Manifolds for hyperbolic times} follows a similar approach to \cite[Theorem 3.5]{AV20}, with the key distinction that we do not rely on invariant manifolds but instead on manifolds exhibiting contraction over a finite number of iterates. Below, we summarize the key steps of the proof.

\begin{proof}[Proof of theorem \ref{Manifolds for hyperbolic times}] By 
{\cite{HPS77} there are a $C^1$ neighborhood $\cV$ of $f$, a number $\rho > 0$, and a continuous family of embeddings}
$$
{\cV \times M \in (g, x) \mapsto \Phi_{g,x} \in \text{Emb}^1(\mathbb{R}^{\dim(E^s)+1}, M)}
$$
{such that for every $g \in \cV$ and every $x \in M$ we have}
\begin{itemize}
    \item $T_x\Phi_{g,x}(\bR^{\dim (E^c)+1}) = E^s_x(g),$ and
    \item $g(\Phi_{g,x}(B_\rho(0))) \subset \Phi_{g,g(x)}(\mathbb{R}^{\dim(E^s)+1}).$
\end{itemize}
    
    
{Consider  $V_g(x)=\Phi_{g,x}(\mathbb{R}^{\operatorname{dim}(E^s)+1})$ and for $c>0$  define $D(x,c,g)=\{y\in V_g(x): d_{V_g(x)}(x,y)<c\}$, where $d_{V_g(x)}$ denote the distance given by the Rimannian structure in $V_g(x)$.}

{Consider $r>0$ as in Corollary \ref{existence of a and V}. 
By the uniform continuity of $\log \|Dg_{|E^{cs}(g,x)}\|$,  there exists $c_1>0$ independent of $g$  such that 
}

$${\|dg_{|E^{cs}(g,x)}\|<e^{-r/2}\|dg_{|E^{cs}(g,y)}\|}$$
    for every $g\in \cV$ and $x,y\in M$ with $d(x,y)<c_1.$

Therefore, if we choose $R_{cs}>0$ sufficiently small  we have that   for every $g\in \cV$ and $(x,n)\in \cG_g$, if $y\in D_g(x,R_{cs})$ then

$${d_{V_g(g^k(x))}(g^kx,k^ky)\leq e^{-rk/2}d_{V_g(x)}(x,y)\,\,  \mbox{ for } 0\leq k<n}.$$
{
Thus, for any $g \in \cV$ and $x \in \cG_g$, the manifold $N^{cs}_x := D(x, R_{cs},g)$ satisfies items $(i)$ and $(ii)$ of the theorem. 
By choosing $R_{cs} > 0$ sufficiently small and, if necessary, reducing the size of   $\cV$,  we can also ensure conditions items $(iii)$ and $(iv)$. This concludes the proof of the theorem.}

\end{proof}

\begin{proposition}\label{g has specification}
Let $\cV$ be a $C^{1+}$ neighborhood of $f$ as Corollary \ref{existence of a and V} and Theorem \ref{Manifolds for hyperbolic times}. Suppose that $g\in \mathcal{V} \cap   \operatorname{Diff}^{1+}(M) $ has unstable foliation $\varepsilon$-minimal. 
   If $\varepsilon>0$ is sufficiently small, then $\cG_g$ has specification at scale $\epsilon(1+K)$ where $K$ is a positive constant.
\end{proposition}

\begin{proof}
    Let $0<\varepsilon<\hat{\varepsilon}$, where $\hat{\varepsilon}$ is as in Theorem \ref{Manifolds for hyperbolic times}, item $(iv)$, and $(x_1,n_1),(x_2,n_2)\in \cG_g$. 
    As $\mathcal{F}^u$ is $\varepsilon$-minimal, there exists $R_u>0$  such that if $D$ is a disk in an unstable manifold of $g$ with an internal radius larger than $\frac{R_u}{2}$ then $D \cap B_\varepsilon(x)\neq \emptyset$ for every $x\in M$.
    
    For $R_u>0$, we fix $M_u$ such that for every unstable disk with a center in $z\in M$ an internal radius less than $R_u$, if $y\in D$ then $d(g^{-M_u}(z),g^{-M_u}(y))\leq \varepsilon$.

    Let $D\subset W^u(g^{n_1+M_u}(x_1))$ a disk with an internal radius larger than $R_u/2$ centred in $g^{n_1+M_u}(x_1)$. 
    By $\varepsilon$-minimality  
    $D\cap B_\varepsilon(x_2)\neq \emptyset$, and by item $(iv)$ of Theorem \ref{Manifolds for hyperbolic times} we can choose  $\hat{y}\in D\cap N^{cs}(g,x_2)$. 
    
    Since $\hat{y}\in D$, we have that
    $$d(g^{-M_u}(\hat{y}),g^{n_1}(x_1))\leq \varepsilon .$$ 
    Analogously, 
    $$d(g^{n_1-k}(x_1),g^{-M_u-k}(\hat{y}))\leq \varepsilon\xi^k \mbox{ for every } k\geq 0.$$ 
    Therefore, $$d(g^j(x_1),g^j(g^{-M_u-n_1}(\hat{y})))\leq \varepsilon\xi^{n_1-j} \text{ for every }0\leq j\leq n_1.$$

    Moreover, as $g^{M_u}(\hat{y})\in N^{cs}_{x_2}$, $$d_{n_2}(g^{M^ u}(\hat{y}), x_2)\leq \varepsilon.$$ Therefore, by \cite[Lemma 5.10]{climThom2021} choosing $\tau=M_u$, we get that $\mathcal{G}_g$ has specification  at scale $\varepsilon K$ where $K$ is a positive constant that can be choosing uniformly in $g$.
\end{proof}
\begin{corollary}\label{specification small}

Given  $\delta>0$ there exists  a $C^1$ neighborhood $\cV$ of $f$,  such that if $g\in \mathcal{V}$ then $g$ has specification on $\cG_g$ at scale $\delta$.
\end{corollary}

\begin{proof}

Given $\delta>0$, $\varepsilon$ be sufficiently small such that $\varepsilon(1+ K)< \delta$, where $K$ is as in Proposition \ref{g has specification}.
 
  Since $\cF^u_f$ is minimal,  we can choose a $C^1$ neighborhood $\mathcal{V}$ of $f$ such that  $\cV$ is as in Proposition \ref{g has specification} and  if $g\in \cV$ then $g$ has unstable foliation  $\varepsilon$-minimal. 
  Therefore,  by Proposition \ref{g has specification} we conclude that every $g\in \cV$  has specification on $\cG_g$ at scale $\delta$.  
\end{proof}

\subsection{Obstruction of expansivity}

\begin{proposition}\label{obstrucao-de-expansividade-pequena} There exist $\varepsilon_0>0$ and $\cU$ neighborhood of $f$ such that if $g\in \cU$,
    $P^{\perp}(g,\phi,\varepsilon_1)<P(g,\phi)$.
\end{proposition}
\begin{proof}
{In \cite[Theorem 3.1]{LVY13}  it was shown that there exist a $C^1$ neighborhood $\cU$ and $\varepsilon_1>0$ such that $h(g, \Gamma_{\varepsilon_0}(g,x))=0$ for every $x\in M$.
 More specifically,  it was proved that if $\mu\in  \cM_g^e$ does not have  zero Lyapunov exponents,  then for $\mu$-almost every $x\in M$, we have  $\Gamma_{\varepsilon_0}(g,x))=\{x\}$. 
 This implies that the set of ergodic measures which are not almost expansive is contained in the set of ergodic measures with zero central Lyapunov exponent.
 Thus, by equation \eqref{hPP<Pto} we conclude that $P^{\perp}(g,\phi,\varepsilon_0)<P(g,\phi)$ for every $g\in \cU$.}
\end{proof}

\subsection{ The set $\cP_g$ has small pressure}
 Now, to use  Theorem \ref{criterioCT} we need to prove that $\cP_g$ has a small pressure.
This is the content of the next result.

\begin{proposition}\label{P has entroy small}
    There exists $\varepsilon_1>0$  such that  $P(\cP_g,g,\phi,\delta,\varepsilon)<P(g,\phi)$ for every $0<\varepsilon<\varepsilon_1$ and $g\in \cV$, where $\cV$ is as in Corollary \ref{existence of a and V}.
\end{proposition}
\begin{proof}

    Let $E_n\subset (\cP_{g})_{n}$ be a $(n,\delta)-$separated set with $\delta>0$  such that $\Lambda(\mathcal{P}_g,g,\phi,\delta,n)=\sum_{x\in E_n}e^{\Phi_0(x,n)}$.
    Consider 
    $$\nu_n=\frac{1}{\sum_{x\in E_n}e^{\Phi_0(x,n)}}\sum_{x\in E_n}e^{\Phi_0(x,n)}\delta_x,
    \,\, \mbox{and}\,  \,\,\mu_n=\frac{1}{n}\sum_{k=0}^{n-1} g^k_*\nu_k. $$
 
Reasoning as in \cite[Theorem 9.10]{walters2000},  we obtain that  any limit point $\mu$ of $\mu_n$ is $g-$invariant  and  $P(\mathcal{P}_g,g,\phi,\delta)\leq P_\mu(g,\phi)$. 
Since $\cP_g$ is as in (\ref{def-Pg}) we get 
 $\lambda^c(\mu)=\int \varphi^c\geq -r$. 
 Thus, taking $a$ as in (\ref{hPP<Pto}) we obtain
 \begin{equation}\label{eqpressure}
     P(\cP_g,g,\phi,\delta)<a<P(g,\phi).
 \end{equation}  
 
 Since $\phi$ is uniformly continuous  there exists $\varepsilon_1$ such that 

\begin{equation*}
     P(\cP_g,g,\phi,\delta,\varepsilon)<a<P(g,\phi) \text{ for every } \varepsilon<\varepsilon_1 \text{ and } g\in \cV.
\end{equation*}
This finishes the proof.
\end{proof}

\subsection{ Bowen's property}
Here we prove that if $\phi$ is H\"older continuous,  then for
every $g \in \cV$, $\phi$ satisfies the Bowen property on $\mathcal{G}_g$.

\begin{proposition}\label{G has bowen property} Let $\cV$ and $\hat{\varepsilon}>0$ as in Corollary \ref{existence of a and V} and Theorem \ref{Manifolds for hyperbolic times} respectively. If $g \in \cV\cap \Diff^+(M)$ then  any potential  H\"older continuous $\phi:M\to \mathbb{R}$ has the Bowen property at any scale less than $\hat{\varepsilon}$ on $\cG_g$.
\end{proposition}
\begin{proof}
Given $g\in \cV \cap \Diff^+(M)$, for every $(x,n)\in \mathcal{G}_g$ consider $N^{cs}(g,x)$ as in  Theorem \ref{Manifolds for hyperbolic times}. Given any $(x,n)\in \mathcal{G}_g$ and $y\in M$ with $d(x,y)<\hat{\varepsilon}$, the intersection of $N^{cs}(g,x)$ with $W^u_{loc}(y)$ is a unique point $\{[x,y]\}$.

\begin{claim}
If $(x, n) \in \cG_g$ and $y \in B_n(x, \varepsilon)$, then
$$
d(g^k(x), g^k(y)) \leq 2 \max\left\{C e^{-rk/2} \delta_0, \xi^{n-k} \delta_0 \right\}, \mbox{ for every } 0\leq k \leq n
$$
where $\xi$ is as defined in Definition \ref{def-partialh}, $C > 0$ is as given in item $(iv)$ of Theorem \ref{Manifolds for hyperbolic times}, and $\delta_0$ is a positive constant that depends on  $\cV$.

\end{claim}
    \begin{proof}  As  $[x,y]\in N^{cs}_{Rc}(x)$, for every $k\in \{0,1,\cdots,n-1\}$ we have

$$d(g^k(x),g^k([x,y])) \leq C{e^{-rk/2}}d(x,[x,y]),$$
since $[x,y]\in W^u(y)$, 
we get
 $$d(g^ky,g^k([x,y])\leq \xi^{n-k}\delta_0,\quad k\in \{0,1,\cdots,n-1\},$$
where $\delta_0> \varepsilon>0$ is such that $d(z,[z,\hat{z}])$ and $d(\hat{z},[z,\hat{z}])$ are both less than $\delta_0$ for every $z\in \cG_g$ and $\hat{z}\in B(x,\hat{\varepsilon})$.

By the triangle inequality, we get
$$d(g^k(x),g^k(y))\leq 2\max\{C{e^{-rk/2}}\delta_0,{\xi^{n-k}}\delta_0\}.$$
\end{proof}

Now, writing $C_0$ for the H\"older constant of $\phi$ and $\gamma$ for the H\"older exponent, we obtain 

$$
\begin{aligned}
\left|\phi\left(g^k x\right)-\phi\left(g^k y\right)\right| & \leq C_0 \cdot d\left(g^k x, g^k y\right)^\gamma\\
& \leq C_0 \cdot (2 \delta_0)^\gamma\max\{C{e^{-\gamma rk/2}},{\xi^{\gamma(n-k)}}\},
\end{aligned}
$$
and summing over $0 \leq k<n$ gives
$$
\begin{aligned}
\left|\Phi_0(x,n)-\Phi_0(y,n)\right| & \leq \sum_{k=0}^{n-1} C_0 \cdot (2 \delta_0)^\gamma \max\{C{e^{-\gamma rk/2}},{\xi^{\gamma(n-k)}}\}, \\
& \leq C_0 \cdot (2 \delta_0)^\gamma \sum_{k=0}^{n-1} C{e^{-\gamma rk/2}}+{\xi^{\gamma(n-k)}}=: K .
\end{aligned}
$$
Since $K$ is finite and independent  of $x,y,n,$ we finish the proof.
\end{proof}

\subsection{ Proof of Theorem \ref{teo-main}}
Let $f:M\to \bR$ and $\phi:M\to \bR$ be as in Theorem \ref{teo-main}. 
Let $\cV\subset  \operatorname{Diff}^{1+}(M)$ and $a>0$ as in Corollary \ref{existence of a and V}. So, every $g\in \cV \cap \Diff^{1+}(M)$ has a decomposition  as defined in Section \ref{choosing descomposition}.

Now, take $\varepsilon$ such that $\varepsilon<\varepsilon_0,\varepsilon_1,\tilde{\varepsilon}$, where $\varepsilon_0$, $\varepsilon_1$ and $\tilde{\varepsilon}$ are as in  propositions \ref{obstrucao-de-expansividade-pequena}, \ref{P has entroy small} and \ref{G has bowen property} respectively. 
Therefore,

\begin{equation}\label{Part 1}
    P(\cP_g,g,\phi,\delta,\varepsilon)<P(g,\phi) \text{ for every } g\in \cV\cap \Diff^{1 +}(M) \text{ and } \delta>0,
\end{equation}
\begin{equation}\label{Part 2}
    P^\perp(g,\phi,\varepsilon)<P(g,\phi)\text{ for every } g\in \cV\cap \Diff^{1 +}(M), 
\end{equation}
and
\begin{equation}\label{part 3}
    \cG_g \text{ has the Bowen's property  in scale } \varepsilon \text{ for every } g\in \cV\cap \Diff^{1 +}(M). 
\end{equation}  

By Proposition \ref{g has specification} there exists $\cU\subset\cV$ such that  every $g\in \cU \cap \Diff^{1 +}(M)$  has specification on $\cG_g$ at scale   $\delta$ with $\delta<2000\varepsilon$. Therefore, by (\ref{Part 1}), (\ref{Part 2}) and (\ref{part 3}) and Theorem \ref{criterioCT} we conclude that $(g,\phi)$  has a unique equilibrium state, finishing the proof.
\vspace{0.2cm}

\section{The Derived from Anosov Example}\label{s-example}

In this section, we show that the classical class of examples of a non-hyperbolic, robustly transitive diffeomorphism, isotopic to an Anosov automorphism and derived from an Anosov system, as constructed by Ma\~n\'e in \cite{M78}, satisfies the hypotheses of our theorem. To do so we proceed as follows.

Let $A$ be a linear Anosov diffeomorphism over the 3-torus $\mathbb{T}^3$ with eigenvalues 
$$0<\lambda_s<1<\lambda_c<\lambda_u.$$ 
A $C^r \, (r\geq 1)$ partially hyperbolic diffeomorphism $f:\bT^3 \to \bT^3$ is called a {\em Derived from Anosov ($DA$, for short) diffeomorphism} if it is in the isotopic class of $A$. 
Let ${\cD}^r(A)$  denote the set of $C^r \, DA$-diffeomorphisms.

By \cite[Theorem 1]{Fra69}, for every $DA$-diffeomorphism $f\in \mathcal{D}^1(A)$, there exists a continuous surjective map $h:\bT^3\to \bT^3$ that semiconjugates $f$ to $A$, meaning $h\circ f=A\circ h$.

\begin{definition}\label{d-coherent}
A partially hyperbolic diffeomorphism $f: M \to M$ is dynamically coherent if 
$E^{cs} := E^c \times E^s$ and $E^{cu} := E^c \times E^u$ are uniquely integrable to invariant
foliations $\cF^{cs}$ and $\cF^{cu}$, respectively the center stable and center unstable foliations.
Particularly $E^c$ is uniquely integrable to the center foliaiton $\cF^c$, which is obtained by
the intersection $\cF^{cs} \cap  \cF^{cu}$.
\end{definition}


A foliation $\cF$ is quasi-isometric if there exists a constant $Q > 1$ such that, for any two points $x, y \in M$  lying on the same leaf of $\cF$, the leafwise distance satisfies $d_{\cF}(x,y) \leq  Qd_M(x,y)+Q$, where $d_{\cF}$ is the distance along the leaf, and $d_M$ is the distance on the manifold.

Let $\pi:\bR^3\to \bT^3$ be the natural lift of $\bT^3$ to its universal cover $\bR^3$ and denote by $\tilde{f}$, $\tilde{h}$, $\tilde{\cF^s}$, $\tilde{\cF^u}$, $\tilde{W^s}$, $\tilde{W^u}$ and $\tilde{A}$ the lift of $f$, $h$, $\cF^s$, $\cF^u$ $W^s$, $W^u$ and  $A$ respectively. 

\begin{proposition}
    Suppose $f\in \cD^1(A)$. then f is dynamically coherent, and the Frank's semiconjugacy $h$ maps the strong stable, center stable, center and center unstable leaves into the corresponding leaves of $A$. Moreover, the foliations  $\tilde{\cF}^u$ and $\tilde{\cF}^s$ are quasi isometric and there is a positive $K \in \bR$ such that $\tilde{h}$ is $K$-close to the identity in $\bR^3$.  
\end{proposition}

  These results are well known and the proof can be found in \cite{Po15,Fra69}, see too \cite{BBI09,HP14} 
  
  
  \begin{theorem}\label{for-DA-hs>hu} 
    If $f\in \cD^1(A)$ with $\operatorname{dim}E^c=1$ then $h^u(f)<h^s(f).$
\end{theorem}

\begin{proof}
    We will prove that $h^u(f)\leq \log(\lambda_u)$. To do so, it suffices to show that for any two points $x$ and $y$ on the same unstable manifold $W^u$ it holds

    $$\limsup_{n\to \infty}\frac{1}{n}\log d_W(f^nx,f^ny)\leq \log(\lambda_u).$$

    Let $x,y\in M$ lie on the same unstable manifold $W^u_z$ of the point $z\in M$
and $\tilde{W}^u_z$, $\tilde{f}$, $\tilde{x}$, $\tilde{y}$ denote the lifts of $W^u_z$, $f$, $x$, $y$ respectively. Then we have 

$$d_{W^u_z}(f^nx,f^ny)=d_{\tilde{W}^u_z}(\tilde{f}^n\tilde{x},\tilde{f}^n\tilde{y}).$$

As $\tilde{W}^u_z$ is quasi isometric, there is a positive number $Q$ such that

$$d_{W^u_z}(f^nx,f^ny)\leq Qd_{\bR^3}(\tilde{f}^n\tilde{x},\tilde{f}^n\tilde{y})+Q.$$

Since $|\tilde{h}(w)-w|<K$ for every $w\in \bR^3$ we get

\begin{equation*}d_{W^u_z}(f^nx,f^ny)\leq -2 K Q +Qd_{\bR^3}(\tilde{\pi}\tilde{f}^n\tilde{x},\tilde{\pi}\tilde{f}^n\tilde{y})+Q,
\end{equation*}

as $\tilde{h} \circ \tilde{f}^n=\tilde{A}^n\circ \tilde{h}$, we have

$$d_{W^u_z}(f^nx,f^ny)\leq -2 K Q +Qd_{\bR^3}(\tilde{A}^n\tilde{\pi}\tilde{x},\tilde{A}^n\tilde{\pi}\tilde{y})+Q,$$

therefore,

$$\limsup_{n\to \infty}\frac{1}{n}\log d_{W^u_z}(f^nx,f^ny)\leq \log \lambda_u,$$

implying that $h^u(f) \leq \log \lambda_u = h^u(A)$.

Moreover, since $h^u(f)\leq \log \lambda_u=h^u(A)< h_{top}(A)\leq h_{top}(f)$ and  $\operatorname{dim}(E^c)=1$, we have that $h^s(f)=h_{top}(f)> h^u(f)$.
\end{proof}

Note that Theorem \ref{teo-main} also holds when the conditions  $h^u(f)> h^s(f)$ and minimality of the unstable foliation are replaced by $h^s(f)> h^u(f)$ and minimality of the stable foliation. To obtain this, one simply applies Theorem \ref{teo-main} to $f^{-1}$. 

\begin{theorem}

Let \( A \) be a linear Anosov diffeomorphism of \( \mathbb{T}^3 \) with eigenvalues \( 0 < \lambda_s < 1 < \lambda_c < \lambda_u \). Let \( f \in D^2(A) \) and suppose the set 
\[
B(f) = \{x : |\det(Tf|_{E^{cu}(x)})| \leq \lambda_u \}
\]
has zero leaf volume inside any strong unstable leaf. Then there is a $C^{1}$ neighborhood $\cU$ of $f$  such that every $C^{1+}$ diffeomorphism $g\in \cU$ has a unique measure of maximal entropy. 

\end{theorem}

\begin{proof}
     Theorem \ref{for-DA-hs>hu} implies that $h^s(f)>h^u(f),$ and  \cite[Theorem A]{HUY22} establishes the stable manifold of $f$ is minimal. 
     Thus, by Theorem \ref{teo-main} there is a $C^1$ neighborhood $\cU$ of $f$ such that if $g\in \cU \cap \Diff^{1+}(M) $ then $g$ has a unique measure of maximal entropy.
\end{proof}

{\em{Acknowledgements.}} We are thankful to Jiagang Yang, Fan Yang and Mauricio Poletti for helpful conversations on this work. We also thank the anonymous referees for their suggestions, which substantially improved the text.

\bibliographystyle{alpha}

\begin{thebibliography}{RHUY22}

\bibitem[AV20]{AV20}
Martin Andersson and Carlos V{\'a}squez.
\newblock Statistical stability of mostly expanding diffeomorphisms.
\newblock In {\em Annales de l'Institut Henri Poincar{\'e} C, Analyse non
  lin{\'e}aire}, volume~37, pages 1245--1270. Elsevier, 2020.

\bibitem[Bow71]{Bowen71}
Rufus Bowen.
\newblock Entropy for group endomorphisms and homogeneous spaces.
\newblock {\em Transactions of the American Mathematical Society},
  153:401--414, 1971.

\bibitem[Bow72]{Bowen72}
Rufus Bowen.
\newblock Entropy-expansive maps.
\newblock {\em Transactions of the American Mathematical Society},
  164:323--331, 1972.

\bibitem[Bow74]{Bowen74}
Rufus Bowen.
\newblock Some systems with unique equilibrium states.
\newblock {\em Mathematical systems theory}, 8(3):193, 1974.

\bibitem[BBI09]{BBI09}
Michael Brin, Dmitri Burago, and Sergey Ivanov.
\newblock Dynamical coherence of partially hyperbolic diffeomorphisms of the
  3-torus.
\newblock {\em J. Mod. Dyn}, 3(1):1--11, 2009.

\bibitem[BP73]{BP73}
Michael Brin and Yakov Pesin.
\newblock Partially hyperbolic dynamical systems.
\newblock {\em Uspehi Mat. Nauk}, 28(3(171)):169--170, 1973.

\bibitem[Bro22]{Br22}
Aaron Brown.
\newblock Smoothness of stable holonomies inside center-stable manifolds.
\newblock {\em Ergodic Theory Dynamical Systems}, 42(12):3593--3618, 2022.

\bibitem[CFT18]{CFT2018}
Vaughn Climenhaga, Todd Fisher, and Daniel~J. Thompson.
\newblock Unique equilibrium states for Bonatti-Viana diffeomorphisms.
\newblock {\em Nonlinearity}, 31(6):2532, 2018.

\bibitem[CFT19]{CFT2019}
Vaughn Climenhaga, Todd Fisher, and Daniel~J. Thompson.
\newblock Equilibrium states for Ma{\~n}{\'e} diffeomorphisms.
\newblock {\em Ergodic theory and dynamical systems}, 39(9):2433--2455, 2019.

\bibitem[CT16]{climThom2016}
Vaughn Climenhaga and Daniel~J. Thompson.
\newblock Unique equilibrium states for flows and homeomorphisms with
  non-uniform structure.
\newblock {\em Advances in Mathematics}, 303:745--799, 2016.

\bibitem[CT21]{climThom2021}
Vaughn Climenhaga and Daniel~J. Thompson.
\newblock Beyond Bowen specification property.
\newblock In {\em Thermodynamic Formalism}, pages 3--82. Springer, 2021.

\bibitem[CY05]{CY05}
William Cowieson and Lai-Sang Young.
\newblock SRB measures as zero-noise limits.
\newblock {\em Ergodic Theory and Dynamical Systems}, 25(4):1115--1138, 2005.

\bibitem[DFPV12]{DFPV12}
Lorenzo~J. D\'{\i}az, Todd Fisher, Maria~Jos\'{e} Pacifico, and Jos\'{e}~L.
  Vieitez.
\newblock Entropy-expansiveness for partially hyperbolic diffeomorphisms.
\newblock {\em Discrete Continuous Dynamical Systems}, 32(12):4195--4207, 2012.

\bibitem[Fra69]{Fra69}
John Franks.
\newblock Anosov diffeomorphisms on tori.
\newblock {\em Transactions of the American Mathematical Society},
  145:117--124, 1969.

\bibitem[Gel16]{Gel16}
Katrin Gelfert.
\newblock Horseshoes for diffeomorphisms preserving hyperbolic measures.
\newblock {\em Mathematische Zeitschrift}, 283:685--701, 2016.

\bibitem[HP14]{HP14}
Andy Hammerlindl and Rafael Potrie.
\newblock Pointwise partial hyperbolicity in three-dimensional nilmanifolds.
\newblock {\em Journal of the London Mathematical Society}, 89(3):853--875,
  2014.

\bibitem[HPS70]{HPS77}
Morris Hirsch, Charles Pugh, and Michael Shub.
\newblock Invariant manifolds.
\newblock {\em Bulletin of the American Mathematical Society},
  76(5):1015--1019, 1970.

\bibitem[HHW17]{HHHY2017}
Huyi Hu, Yongxia Hua, and Weisheng Wu.
\newblock Unstable entropies and variational principle for partially hyperbolic
  diffeomorphisms.
\newblock {\em Advances in Mathematics}, 321:31--68, 2017.

\bibitem[HSX08]{HYS08}
Yongxia Hua, Radu Saghin, and Zhihong Xia.
\newblock Topological entropy and partially hyperbolic diffeomorphisms.
\newblock {\em Ergodic Theory Dynam. Systems}, 28(3):843--862, 2008.

\bibitem[Kat80]{K80}
Anatole Katok.
\newblock Lyapunov exponents, entropy and periodic orbits for diffeomorphisms.
\newblock {\em Inst. Hautes \'{E}tudes Sci. Publ. Math.}, (51):137--173, 1980.

\bibitem[KH95]{KH95}
Anatole Katok and Boris Hasselblatt.
\newblock {\em Introduction to the modern theory of dynamical systems},
  volume~54 of {\em Encyclopedia of Mathematics and its Applications}.
\newblock Cambridge University Press, Cambridge, 1995.
\newblock With a supplementary chapter by Katok and Leonardo Mendoza.

\bibitem[LVY13]{LVY13}
Gang Liao, Marcelo Viana, and Jiagang Yang.
\newblock The entropy conjecture for diffeomorphisms away from tangencies.
\newblock {\em Journal of the European Mathematical Society (JEMS)}, 15(6):2043--2060, 2013.

\bibitem[LY85a]{LYII85}
Fran\c cois Ledrappier and Lai-Sang Young.
\newblock The metric entropy of diffeomorphisms. {I}. {C}haracterization of
  measures satisfying {P}esin's entropy formula.
\newblock {\em Annals of Mathematics (2)}, 122(3):509--539, 1985.

\bibitem[LY85b]{LYI85}
Fran\c cois Ledrappier and Lai-Sang Young.
\newblock The metric entropy of diffeomorphisms. {II}. {R}elations between
  entropy, exponents and dimension.
\newblock {\em Annals of Mathematics (2)}, 122(3):540--574, 1985.


\bibitem[M78]{M78}
Ricardo Ma\~n\'e. 
\newblock Contributions to the stability conjecture. 
\newblock {\em Topology, 17(4)}, 383--396, 1978.


\bibitem[Mar]{mar70}
Grigoriy Margulis.
\newblock On some aspects of the theory of anosov systems with a survey by
  Richard Sharp: Periodic orbits of hyperbolic flows.
  \newblock {\em Springer Monographs in Mathematics}, 1970.

\bibitem[New89]{N89}
Sheldon Newhouse.
\newblock Continuity properties of entropy.
\newblock {\em Annals of Mathematics (2)}, 129(2):215--235, 1989.

\bibitem[Pot15]{Po15}
Rafael Potrie.
\newblock Partial hyperbolicity and foliations in $\mathbb{T}^3$.
\newblock {\em Journal of Modern Dynamics}, 9(1):81--121, 2015.

\bibitem[PYY22a]{PYY2022}
Maria Jos\'e Pacifico, Fan Yang, and Jiagang Yang.
\newblock Existence and uniqueness of equilibrium states for systems with
  specification at a fixed scale: an improved Climenhaga-Thompson criterion.
\newblock {\em Nonlinearity}, 35(12):5963, 2022.

\bibitem[PYY22b]{PYYLorenz2022}
Maria~Jose Pacifico, Fan Yang, and Jiagang Yang.
\newblock Uniqueness of equilibrium states for Lorenz attractors in any
  dimension.
\newblock {\em arXiv preprint arXiv:2201.06622}, 2022.

\bibitem[RHUY22]{HUY22}
Jana Rodriguez~Hertz, Ra{\'u}l Ures, and Jiagang Yang.
\newblock Robust minimality of strong foliations for da diffeomorphisms:
  $cu$-volume expansion and new examples.
\newblock {\em Transactions of the American Mathematical Society},
  375(6):4333--4367, 2022.

\bibitem[Roh49]{RO49}
Vladimir Rohlin.
\newblock On the fundamental ideas of measure theory.
\newblock {\em  Mat. Sb.(NS), 25(67):107--150, 1949;
American Mathematical Society Translations}, 71 (1952).


\bibitem[Rue68]{Ruelle68}
David Ruelle.
\newblock Statistical mechanics of a one-dimensional lattice gas.
\newblock {\em Communications in Mathematical Physics}, 9:267--278, 1968.

\bibitem[Rue78]{Ruelle78}
David Ruelle.
\newblock Thermodynamic formalism (encyclopedia of mathematics and its
  applications, 5).
\newblock {\em Addisan-Wesley Publicating Company}, 1978.

\bibitem[Sin72]{sinai72}
Yakov Sinai.
\newblock Gibbs measures in ergodic theory.
\newblock {\em Russian Mathematical Surveys}, 27(4):21, 1972.

\bibitem[SVY16]{SNVPY2016}
Naoya Sumi, Paulo Varandas, and Kenichiro Yamamoto.
\newblock Partial hyperbolicity and specification.
\newblock {\em Proceedings of the American Mathematical Society},
  144(3):1161--1170, 2016.

\bibitem[SX10]{SX10}
Radu Saghin and Zhihong Xia.
\newblock The entropy conjecture for partially hyperbolic diffeomorphisms with
  1-{D} center.
\newblock {\em Topology Appl.}, 157(1):29--34, 2010.

\bibitem[Wal00]{walters2000}
Peter Walters.
\newblock {\em An introduction to ergodic theory}, volume~79.
\newblock Springer Science \& Business Media, 2000.

\bibitem[Yan21]{yang2016}
Jiagang Yang.
\newblock Entropy along expanding foliations.
\newblock {\em Advances in Mathematics}, 389:107893, 2021.

\bibitem[Yom87]{Y87}
Yosef Yomdin.
\newblock Volume growth and entropy.
\newblock {\em Israel J. Math.}, 57(3):285--300, 1987.

\end{thebibliography}

\end{document}